\documentclass[12pt]{article}
\usepackage{amssymb}
\usepackage{amsmath,bm}
\usepackage{graphics,mathrsfs}
\usepackage{graphicx,epsfig}
\usepackage{subfigure}
\usepackage{setspace}
\usepackage{enumerate}
\usepackage{enumitem}
\usepackage{varioref}
\usepackage{natbib}
\setcitestyle{authoryear,open={(},close={)}}
\usepackage{amsthm}
\makeatletter \oddsidemargin  -.1in \evensidemargin -.1in
\begin{document}
\newcommand{\bea}{\begin{eqnarray}}
\newcommand{\eea}{\end{eqnarray}}
\newcommand{\nn}{\nonumber}
\newcommand{\bee}{\begin{eqnarray*}}
\newcommand{\eee}{\end{eqnarray*}}
\newcommand{\lb}{\label}
\newcommand{\nii}{\noindent}
\newcommand{\ii}{\indent}
\newtheorem{Identity}{Identity}[section]
\newtheorem{example}{Example}[section]
\newtheorem{counterexample}{Counterexample}[section]
\newtheorem{corollary}{Corollary}[section]
\newtheorem{definition}{Definition}[section]
\newtheorem{lemma}{Lemma}[section]
\newtheorem{remark}{Remark}[section]
\newtheorem{proposition}{Proposition}[section]
\renewcommand{\theequation}{\thesection.\arabic{equation}}
\renewcommand{\labelenumi}{(\roman{enumi})}
%\renewcommand*\thesubsection{\arabic{subsection}}
%\renewcommand{\thesection}{\arabic{section}}
%\renewcommand{\theequation}{\thesubsection.\arabic{equation}}
%\renewcommand{\theequation}{2.\arabic{equation}}
%\doublespacing
\title{\bf A probabilistic proof of the finite geometric series}
\author{Raju Dey and Suchandan Kayal\thanks {Email address (corresponding author):
                 kayals@nitrkl.ac.in,~suchandan.kayal@gmail.com}
\\{\it \small Department of Mathematics, National Institute of
Technology Rourkela, Rourkela-769008, India}}
%%\title{\bf Ordering for series and parallel systems
%%associated with heterogeneous GMW components}
%%\author{Raju Dey and Suchandan Kayal\thanks {Email address (corresponding author):
%%                    kayals@nitrkl.ac.in,~suchandan.kayal@gmail.com}
%%                    %{\it \small $^a$Dept. of Mathematics and Statistics,
%%%McMaster University, Hamilton, Ontario L8S4K1, Canada}
%%\\{\it \small Department of Mathematics, National Institute of
%%Technology Rourkela, Rourkela-769008, India}}
\date{}
\maketitle
\begin{center}
{\large \bf Abstract}
\end{center}
In this note, we present a probabilistic proof of
the well-known finite geometric series. The
proof follows by taking the moments of the sum
and the difference of two independent
exponentially distributed random variables.

\section{Introduction} There are various popular series
in the literature related to infinite series. One
of these is the geometric series. Each term of
the geometric series is a constant multiple of
the the previous term. For example, if $a$ is the
first entry of the series, then the second entry
is a constant (a real number, say $\rho\ne0$)
multiple of $a$, the third entry is (same)
constant multiple of $a\rho$, and so on. The real
number $\rho$ is known as the common ratio of the
series. The infinite geometric
series is
\begin{eqnarray}
a+a\rho+a\rho^2+a\rho^3+\cdots.
\end{eqnarray}
Without loss of generality, here, we assume that
$a=1$. Thus, the infinite geometric series
becomes
\begin{eqnarray}\label{eq1.2}
1+\rho+\rho^2+\rho^3+\cdots.
\end{eqnarray}
The partial sum of the first $n$ terms of the
infinite geometric series given by
$(\ref{eq1.2})$ is
\begin{eqnarray}\label{eq1.3}
S_{n}=\sum_{k=0}^{n-1}\rho^{k}=\frac{1-\rho^n}{1-\rho}.
\end{eqnarray}
Note that the formula given in (\ref{eq1.3}) is
also known as the geometric identity for finite
number of terms. In this letter, we establish the
identity given by (\ref{eq1.3}) from a
probabilistic point of view, which is provided in
the next section.

\section{Proof of (\ref{eq1.3})\setcounter{equation}{0}}
To establish the sum formula for the finite
geometric series, the following lemma is
useful. Let $X$ be an exponentially distributed
random variable with probability density function
\begin{eqnarray}\label{eq2.1}
f_{X}(x|\lambda) = \left\{\begin{array}{ll}
\displaystyle\lambda e^{-\lambda x},
& \textrm{if $x>0,$}\\
0,& \textrm{otherwise,}
\end{array} \right.
\end{eqnarray}
where $\lambda>0$. If $X$ has the density
function given by (\ref{eq2.1}), then for
convenience, we denote $X\sim Exp(\lambda)$.
\begin{lemma}
Let $X$ and $Y$ be two independent random
variables such that $X\sim Exp(\lambda)$ and
$Y\sim Exp(\mu)$, where $\lambda\ne\mu$. Further,
let $U=X+Y$ and $V=X-Y$. Then,
\begin{eqnarray}\label{eq2.2}
f_{U}(u|\lambda,\mu) = \left\{\begin{array}{ll}
\displaystyle\frac{\lambda\mu}{\lambda-\mu}~\left(e^{-\mu
u}- e^{-\lambda u}\right),
& \textrm{if $u>0,$}\\
0,& \textrm{otherwise}
\end{array} \right.
\end{eqnarray}
and
\begin{eqnarray}\label{eq2.3}
f_{V}(v|\lambda,\mu) = \left\{\begin{array}{ll}
\displaystyle\frac{\lambda\mu}{\lambda+\mu}~e^{-\lambda
v},
& \textrm{if $v>0,$}\\
\displaystyle\frac{\lambda\mu}{\lambda+\mu}~e^{\mu
v},& \textrm{otherwise.}
\end{array} \right.
\end{eqnarray}
\end{lemma}
\begin{proof}
Consider the transformations $U=X+Y$ and
$U_{1}=X$. Then, the joint probability density
function of $U$ and $U_{1}$ is obtained as
\begin{eqnarray}\label{eq2.4}
f_{U,U_{1}}(u,u_{1})=\lambda \mu
e^{-(\lambda-\mu)u_{1}}e^{-\mu
u},~~0<u<\infty,~0<u_{1}<u.
\end{eqnarray}
Thus, the probability density function of $U$
given by $(\ref{eq2.2})$ can be obtained after
integrating (\ref{eq2.4}) with respect to
$u_{1}$. Similarly, the density function in
(\ref{eq2.3}) can be derived if we take the
transformations $V=X-Y$ and $V_{1}=X$.
\end{proof}

Now, we present the proof of (\ref{eq1.3}). Note
that the $k$th order moments of $X$ and $Y$ about
the origin are respectively given by
$E(X^k)=k!/\lambda^{k}$ and
$E(Y^{k})=k!/\mu^{k}$, where $k=0,1,2,\cdots$ (see Rohatgi and Saleh, 2015).
Further, using (\ref{eq2.2}), it can be shown
that
\begin{eqnarray}\label{eq2.5}
E(X+Y)^{n-1}&=&\int_{0}^{\infty}u^{n-1}\left(\frac{\lambda\mu}{\lambda-\mu}\right)\left(e^{-\mu
u}- e^{-\lambda u}\right)du\nonumber\\
&=&\frac{(n-1)!(\lambda\mu)
}{\lambda-\mu}\left(\frac{1}{\mu^{n}}-\frac{1}{\lambda^{n}}\right).
\end{eqnarray}
The
binomial theorem in (\ref{eq2.5}) yields
\begin{eqnarray}\label{eq2.6}
E\left(\sum_{k=0}^{n-1}{n-1\choose k}X^{k}Y^{n-1-k}\right)&=&\frac{(n-1)!(\lambda\mu)
}{\lambda-\mu}\left(\frac{1}{\mu^{n}}-\frac{1}{\lambda^{n}}\right)\nonumber\\
\Rightarrow \sum_{k=0}^{n-1}{n-1\choose k}E(X^{k})E(Y^{n-1-k})&=&\frac{(n-1)!(\lambda\mu)
}{\lambda-\mu}\left(\frac{1}{\mu^{n}}-\frac{1}{\lambda^{n}}\right)\nonumber\\
\Rightarrow \sum_{k=0}^{n-1}{n-1\choose k} \frac{k!}{\lambda^{k}}\frac{(n-1-k)!}{\mu^{n-1-k}}&=&\frac{(n-1)!(\lambda\mu)
}{\lambda-\mu}\left(\frac{1}{\mu^{n}}-\frac{1}{\lambda^{n}}\right)\nonumber\\
\Rightarrow \sum_{k=0}^{n-1} \left(\frac{\mu}{\lambda}\right)^{k}&=&\frac{\lambda \mu^{n}}{\lambda-\mu}\left(
\frac{1}{\mu^{n}}-\frac{1}{\lambda^{n}}\right)\nonumber\\
&=& \frac{1}{1-\frac{\mu}{\lambda}}\left(1-\left(\frac{\mu}{\lambda}\right)^{n}\right)
\end{eqnarray}
Thus, the identity given by (\ref{eq1.3}) is
established when $\rho=\mu/\lambda$ is a strictly
positive real number. To prove the identity when
$\rho=\mu/\lambda<0$, we consider the expectation
of $(X-Y)^{n-1}$, where the probability density
function of $V=X-Y$ is given by (\ref{eq2.3}).
Now,
\begin{eqnarray}
E(X-Y)^{n-1}&=&\int_{0}^{\infty}u^{n-1}\frac{\lambda \mu}{\lambda+\mu}e^{-\lambda u}du+
\int_{-\infty}^{0}u^{n-1}\frac{\lambda \mu}{\lambda+\mu}e^{\mu u}du\\
&=&\frac{(n-1)!\lambda
\mu}{\lambda+\mu}
\left(\frac{1}{\lambda^{n}}+(-1)^{n-1}\frac{1}{\mu^{n}}\right).
\end{eqnarray}
Using similar arguments to (\ref{eq2.6}), we obtain
\begin{eqnarray}\label{eq2.8}
\sum_{k=0}^{n-1}\left(-\frac{\mu}{\lambda}\right)^{k}=
\frac{1-\left(-\frac{\mu}{\lambda}\right)^{n}}
{1-\left(-\frac{\mu}{\lambda}\right)}.
\end{eqnarray}
Thus, the desired identity is proved when
$\rho=-\mu/\lambda$ is strictly negative.
Combining (\ref{eq2.6}) and (\ref{eq2.8}), the
identity given by (\ref{eq1.3}) is established
for any nonzero real number $\rho$.

\begin{remark}
Taking limit $n\rightarrow \infty$ in
(\ref{eq1.3}), one can prove that
$\sum_{k=0}^{\infty}\rho^{k}=\frac{1}{1-\rho}$
for $|\rho|<1$.
\end{remark}
\noindent
\\
{\bf \large References}\\
Rohatgi, V. K. and Saleh, A. M. E. (2015). An introduction to probability and statistics,
John Wiley \& Sons.
\end{document}